\documentclass[12pt,reqno,a4paper]{amsart}

\usepackage{a4} 
\usepackage{amssymb} 
\usepackage[utf8]{inputenc} 
\usepackage[T1,T2A]{fontenc}
\usepackage{ae} 
\usepackage{tikz} 
\usepackage[russian,english]{babel}
\usepackage[colorlinks,linkcolor=blue]{hyperref}
\usepackage{doi} 

\numberwithin{equation}{section} 
 
\renewcommand{\Bbb}{\mathbb} 
\newcommand{\cal}{\mathcal} 
\newcommand{\wt}[1]{\widetilde{#1}}

\renewcommand{\C}{{\Bbb C}} \newcommand{\N}{{\Bbb N}} 
\renewcommand{\P}{{\Bbb P}} \newcommand{\R}{{\Bbb R}} 
 \newcommand{\Z}{{\Bbb Z}}  
\newcommand{\cA}{{\cal A}}  
\newcommand{\vp}{\varphi} 
\newcommand{\ep}{\varepsilon}

\def\pabi#1/#2{\frac{\partial #1}{\partial x_{#2}}} 
\def\pab#1/#2{\frac{\partial #1}{\partial #2}}

\DeclareMathOperator{\Char}{char} 
\DeclareMathOperator{\Supp}{Supp} 
 
\DeclareMathOperator{\ini}{In}

\newtheorem{theorem}{Theorem}[section] 
\newtheorem{prop}[theorem]{Proposition} 
\newtheorem{lemma}[theorem]{Lemma} 
 
\newtheorem{cor}[theorem]{Corollary} 
\newtheorem{conj}[theorem]{Conjecture}

\theoremstyle{definition} 
\newtheorem{defn}[theorem]{Definition} 
\newtheorem{remark}[theorem]{Remark} 
\newtheorem{example}[theorem]{Example} 
\newtheorem*{acknowledgement}{Acknowledgement} 
 
\begin{document} 
 
\title{Conjectures on stably Newton degenerate singularities}  
 
\author{Jan Stevens} 
\address{Department of Mathematical Sciences, Chalmers University of  
Technology and University of Gothenburg.  
SE 412 96 Gothenburg, Sweden} 
\email{stevens@chalmers.se}

\begin{abstract} 
We discuss a problem of Arnold, whether every function is stably equivalent 
to one which is non-degenerate for its Newton diagram. We argue that the answer 
is negative.
We describe a method to make functions non-degenerate after 
stabilisation and give  examples of singularities 
where this method does not work. 
We  conjecture that they are in fact stably degenerate, that is
not   stably equivalent to non-degenerate 
functions.
  
We review the various non-degeneracy concepts in 
the literature. For finite characteristic we conjecture that there are no
wild vanishing cycles for non-degenerate singularities. This implies that
the simplest example of singularities with  finite Milnor number, 
$x^p+x^q$ in characteristic  
$p$, is not   stably equivalent to a non-degenerate 
function.

We argue that irreducible plane curves with an arbitrary number 
of Puiseux pairs (in characteristic zero) are stably non-degenerate. 
As the stabilisation involves 
many variables, it becomes very difficult to determine the Newton diagram 
in general, but the form of the equations indicates that the defining functions 
are non-degenerate. 
\end{abstract}

\maketitle

\section*{Introduction}\label{s1} 
Many invariants of a hypersurface 
singularity can be computed from its Newton diagram, 
if the singularity is non-degenerate.  Almost all 
singularities with a given diagram are non-degenerate, 
but a given function is degenerate for most choices of coordinates
and for most functions it is even impossible to find suitable coordinates 
in which the function is non-degenerate.
Sometimes this becomes possible after 
adding a quadratic form in new variables to the function.
Invariants computed from the Newton diagram 
of the new function allow conclusions about the original singularity. 
A successful case is the study of Luengo's example 
\cite{lu} of a non-smooth $\mu$-const stratum in \cite{st}. 
Attention to the fact that  a singularity can be made non-degenerate by a 
coordinate transformation after adding 
variables was drawn by Arnold, who raised the  
question whether this is always possible. 
 
Problem 3 of Arnold's list \cite{arn}  in the Arcata volume 
(the Russian version of the problem is older, see problems
1975-3 and 1976-8 in \cite{AP}) reads:
\begin{quote}\it 
Is every function stably equivalent to a $\Gamma$-non-degenerate function 
(in a neighbourhood of a critical point of finite multiplicity)? 
\end{quote}  

Function germs are $R$-equivalent (or shortly equivalent) if they 
can be turned into each other under the action of invertible
coordinate changes, and \textit{stably equivalent} if they become
equivalent after the addition of non-degenerate quadratic forms in
additional variables \cite[11.1]{AGV-I}.
The function $f(x_0,\dots,x_n)+Q(y_0,\dots,y_m)$
with $Q$ a  non-degenerate quadratic form is called a 
\textit{stabilisation}  of the function $f(x_0,\dots,x_n)$.

In this paper we argue that the answer to Arnold's question is negative. We call a function
which is not stably equivalent to a non-degenerate function shortly
for stably degenerate.

The Newton number $\nu(\Gamma)$ of a Newton diagram $\Gamma$ gives a
lower bound for the Milnor number $\mu(f)$
\cite{Kou} and for a non-degenerate function $f$ the 
equality $\mu(f)=\nu(\Gamma(f))$ holds. This equality is a necessary
and sufficient condition for a weaker non-degeneracy condition,
defined by Mondal \cite{Mo-bk}.  His 
\textit{partially non-degeneracy} condition does not involve the partial derivatives 
of initial forms, but initial forms of the partial derivatives.
Another condition (NPND$^*$) was introduced by Wall \cite{Wa}, who 
wanted a condition  sufficient for the
principal results of the theory, and wide enough to include all weighted homogeneous
functions with isolated singularity.
Following the terminology of \cite{BGM} we call it \textit{inner non-degeneracy}.
Conjecturally Mondal's and Wall's conditions are equivalent in
characteristic zero. 

In finite characteristic one can make the same definitions,
but the results are weaker. It is no longer true that the generic function
with a given diagram is non-degenerate. This is related to the occurrence of wild vanishing cycles, see  \cite{SGA}.
We conjecture that these do not appear for non-degenerate singularities. 
This implies that
the simplest example of singularities with  finite Milnor number, 
$x^p+x^q$ in characteristic  
$p$, is not   stably equivalent to a non-degenerate 
function.

This negative answer does not 
extend to the case of real or complex functions. 
I found a number of successful cases, using basically only one trick, which 
however carries a long way. By lack of counterexamples I  expected that every 
function could be made non-degenerate. The first indication that this is not 
true came by considering deformations on the $\mu$-const stratum in Luengo's 
example. A closer analysis led to simpler examples.
The easiest example (see Example \ref{ex23})
is the singularity  
\[ 
f_{23}=x^5+xy^3+z^3-3x^2yz+ x^4y\;.
\]
We conjecture that in fact  $\mu(f)=\mu(\tilde f)>\nu(\Gamma(\tilde f))$
for every $\tilde f$, stably equivalent to $f_{23}$. 
This   implies degeneracy for
all three concepts. We present Luengo's example and other examples,
both stably non-degenerate and conjecturally  stably degenerate.
In particular, we conjecture that there are stably degenerate and
stably non-degenerate functions on the same $\mu$-constant stratum.
This means that simple and less simple topological invariants do not
discriminate between stably degenerate and non-degenerate singularities.

The main reason  to conjecture that our examples are
stably degenerate is that 
our methods do not work in these cases. We describe why they have
to fail. This does not 
exclude the possibility that some unknown, complicated transformation 
makes the function non-degenerate after stabilisation.

In the last section evidence is presented that  every 
irreducible plane curve singularity (with  an arbitrary 
number of Puiseux pairs) is stably equivalent to a non-degenerate 
singularity. 
The number of 
variables is rapidly increasing, making it difficult to 
determine the Newton diagram and check non-degeneracy,
but the form of the equations indicates that the defining functions 
are non-degenerate.

\begin{acknowledgement} 
I thank Claus Hertling for asking the decisive question about non-degeneracy 
of singularities on a non-smooth $\mu$-const stratum. 
\end{acknowledgement} 
 
\section{Non-degenerate functions}\label{s2} 
We recall the standard definition of non-degeneracy, given by  
Kouchnirenko \cite{Kou}, and the related concepts of Wall \cite{Wa}
and Mondal \cite{Mo-bk}. 

\subsection{The Newton diagram}
Let $f\in k[[x_1,\dots,x_n]]$ be a formal power series over  
a field $k$, with  algebraic closure $K$. Write (in multi-index notation) $f=\sum a_mx^m$. The support of $f$ is  $\Supp(f)=\{m\in\N^n\mid a_m\neq0\}$
(note that $0\in \N$). We will assume that $f(0)=0$, so $0\notin \Supp(f)$ 
(otherwise one defines the reduced support by removing the origin \cite[6.2.1]{AGV}).
A Newton diagram $\Gamma(\cA)$ can be defined for an arbitrary subset  $\cA$
of $\N^n$ not containing the origin.  
The Newton diagram $\Gamma(f)$ of $f$ is then the Newton diagram of its support. 
The \textit{Newton polyhedron} $\Gamma_+(\cA)$ is the convex hull of the set 
$\bigcup_{m\in \cA}(m+\R_+^n) \subset \R^n$.
The \textit{Newton diagram} $\Gamma(\cA)$ of $\cA$ is the union 
of all compact faces of $\Gamma_+(\cA)$. The union $\Gamma_-(\cA)$ 
of all segments connecting the origin and the Newton diagram is the \textit{Newton polytope}.  See Figure \ref{fig-1} for an example.

A set $\cA$ is  \textit{convenient} if it contains a point on each coordinate axis.
A series $f$ is convenient if its  support is convenient, that is 
if for every $1\leq i \leq n$ there is a $m_i$ such that the  
monomial $x_i^{m_i}$ occurs with non-zero coefficient. 
In such cases also the Newton diagram is called convenient.

\begin{figure}
\begin{tikzpicture}[punt/.style={circle,fill=blue,inner sep=2pt}]
\draw (1.5,1.5) node {$\Gamma_-(f)$} (5.5,3.5) node {$\Gamma_+(f)$};
\draw[help lines] (0,0) grid (7.5,4.5);
\draw[thick, ->] (0,0) -- (7.6,0) node[right] {$x$}; 
\draw[thick, ->] (0,0) -- (0,4.6) node[above] {$y$};
\foreach \x in {1,...,7} \draw[shift={(\x,0)}]  node[below] {$\scriptstyle\x$};
\foreach \y in {1,2,3,4} \draw[shift={(0,\y)}]  node[left] {$\scriptstyle \y$};
\draw (0,0) node[below left] {$\scriptstyle 0$};
\draw[blue,thick]
node[punt] at (0,4)  {}
(0,4) --node[punt]{} (6,0)  node[punt]{}
node [punt] at (5,1) {}  node [punt] at (7,0) {} ;
\fill[color=gray, opacity=.3] (0,4.5) -- (0,4) -- (6,0) -- (7.5,0) -- (7.5,4.5) ;
\end{tikzpicture}
\caption{The Newton polyhedron $\Gamma_+(f)$ of $f=(y^2-x^3)^2-4x^5y+x^7$.}
\label{fig-1}
\end{figure}
 
Given $f=\sum a_mx^m$ and a
subset $S\subset \R^n$ (e.g., a face $\Delta$ of $\Gamma(f)$)
we denote  by $f_S$ 
the series $\sum_{m\in S}a_mx^m$.  
The \textit{principal part} of $f$  is the polynomial  
$f_{\Gamma}=\sum_{m\in\Gamma(f)}a_mx^m$.
 The classical concept of non-degeneracy is treated by Kouchnirenko \cite{Kou}. 
\begin{defn}\label{def-nondeg} 
The  series $f$ is \textit{non-degenerate} if for every closed 
face $\Delta\subset \Gamma(f)$ the polynomials 
\[ 
x_1\pabi f_\Delta/1,\dots, x_n\pabi f_\Delta/n 
\] 
have no common zero on the torus $(K^*)^n$. 
\end{defn}

\begin{figure}
\begin{tikzpicture}[punt/.style={circle,fill=blue,fill opacity=1,inner sep=2pt}]
\draw[help lines] (0,0) grid (7.5,2.5);
\draw[thick, ->] (0,0,0) -- (7.6,0,0) node[right] {$x$}; 
\draw[thick, ->] (0,0,0) -- (0,2.6,0) node[above] {$y$};
\draw[thick, ->] (0,0,0) -- (0,0,2.6) node[below left] {$z$};
\foreach \x in {1,...,7} \draw[shift={(\x,0)}, help lines] node[above right, black] {$\scriptstyle\x$} (0,0) -- (-.96,-.96) ;
\foreach \y in {1,2} \draw[shift={(0,\y,0)}, help lines]  node[above right,black] {$\scriptstyle \y$} (0,0) -- (-.96,-.96);
\foreach \z in {1,2} \draw[shift={(0,0,\z)}, help lines]  node[above, black] at (\z * 0.05,0.1,0) {$\ \scriptstyle \z$} (0,0) -- (7.5,0) (0,0) -- (0,2.5);
\draw (0,0,0) node[above right] {$\scriptstyle 0$};
\filldraw[blue,thick, fill opacity=.3]
node[punt] at (0,2,1)  {}
(0,2,1) -- (0,0,2) node[punt]{} -- (3,0,1) node[punt]{} -- (7,0,0)   node[punt]{}
 --  (5,1,0)   node[punt]{} -- (0,2,1) -- (3,0,1) -- (5,1,0);
\end{tikzpicture}
\caption{The Newton diagram $\Gamma(\tilde f)$ of $\tilde f=-z^2+2z(y^2-x^3)-4x^5y+x^7$.}
\label{fig-2}
\end{figure}

\begin{example}\label{char13}
The function $f=(y^2-x^3)^2-4x^5y+x^7$ is degenerate. Its Newton diagram
$\Gamma(f)$ can be seen in Figure \ref{fig-1} as the line between
the Newton polytope
$\Gamma_-(f)$ and  the
Newton polyhedron $\Gamma_+(f)$. If $\Char k\neq 2$, the function $f$ is stably equivalent to the 
non-degenerate function (provided $\Char k \neq 3,13$)
$\tilde f=-\tilde z^2+(y^2-x^3)^2-4x^5y+x^7=
-z^2+2z(y^2-x^3)-4x^5y+x^7$ (where $f-\tilde z^2$ is a stabilisation and 
$\tilde z= z-y^2+x^3$ a coordinate transformation) with Newton diagram as shown in
Figure \ref{fig-2}. The function $f$ is convenient, but   $\tilde f$ is not.
\end{example}

\subsection{Milnor and Newton number}
If $f$ is non-degenerate, many invariants can be computed  
from the Newton diagram. We concentrate here on the 
\textit{Milnor number}  
\[
\mu(f)=\dim_k k[[x_1,\dots,x_n]]\big/\big(\textstyle \pabi f/1,\dots, \pabi f/n\big)\;.
\] 
Note  that $\mu(f)$ can be infinite.

For any compact polytope $S$ in $\R_+^n$ with the origin as 
vertex we denote by $V_k(S)$ the sum of the  $k$-dimensional volumes of the intersections of $S$ with the $k$-dimensional 
coordinate subspaces of  $\R^n$, and we define 
following Kouchnirenko its \textit{Newton number} to be 
\[ 
\nu(S) =\sum_{k=0}^n(-1)^{n-k}k!V_k(S)\;. 
\]

\begin{defn}
The \textit{Newton number} $\nu(f)$ of a convenient series 
$f$ is the Newton number of its Newton polytope
$\Gamma_-(f)$.  For  a non convenient 
series $\nu(f):=\sup_{m\in\N}  \nu(f+\sum x_i^{m})$.
\end{defn} 
 Likewise one can define the Newton number $\nu(\cA)$ of a set $\cA$; it is in fact the common value of $\nu(f)$ for all $f$ with $\Supp(f)=\cA$.
 
The main result of Kouchnirenko  \cite{Kou} is: 
 
\begin{theorem}\label{kouchthm} 
For every series $f\in k[[x_1,\dots,x_n]]$ one has $\mu(f)\geq \nu(f)$. Equality 
holds if $f$ is convenient and  non-degenerate.  
If $\Char k=0$, then equality holds also for non-degenerate series which are
not convenient. Moreover, then almost all series with given Newton diagram
are non-degenerate.
\end{theorem}  

Kouchnirenko proves that in characteristic zero the set of degenerate
principal parts is a proper algebraic subset in the variety of all principal
parts corresponding to a given Newton diagram \cite[Th\'eor\`eme I (iii)]{Kou}.
Furthermore, given any subset $\cA\subset \N^n\setminus\{0\}$,
with $\nu(\cA)<\infty$ there exist a non-degenerate series $f$ with
$\Supp(f)=\cA$ \cite[1.13 Remarque (i)]{Kou}, see also \cite{Kou2}
where a combinatorial criterion on $\cA$ for $\nu(\cA)<\infty$ is given.

For nonisolated singularities  the meaning of  $\nu(\Gamma_-(f))$ is 
in the complex case given by  
a theorem of Varchenko \cite{Va}, conjectured by Kouchnirenko \cite{Kou}. 
 
\begin{theorem} 
For a non-degenerate series $f\in \C\{x_1,\dots,x_n\}$ the 
Newton number $\nu(\Gamma_-(f))$ is equal to $(-1)^{n-1}(\chi(F)-1)$, 
where $\chi(F)$ is the Euler characteristic of the Milnor fibre. 
\end{theorem}

The converse of  Theorem \ref{kouchthm} does not hold in general: 
for degenerate series it can be that $\mu(f)= \nu(f)$. 

\begin{example}\label{muisnu}
The simplest example 
is the function $(z+x)^2+xy+y^2$ \cite[Remarque 1.21]{Kou}.
More generally, one can start from any homogeneous isolated curve singularity of the form $yf(x,y)$ of degree $d$ and make a suspension $z^d+yf(x,y)$.
A simple linear coordinate transformation gives the degenerate function 
$g=(z+x)^d+yf(x,y)$, but $\mu(g)=\nu(g)=(d-1)^3$.
\end{example}

\subsection{Inner and partially non-degenerate functions}
The function in the above example is in fact non-degenerate in the sense of
Wall \cite{Wa} and of Mondal \cite{Mo-bk}. As their definitions are
given for algebraically closed fields, we assume from now on for simplicity
that the coefficient field is algebraically closed; the definitions to follow can
easily be extended by taking coefficients in a smaller field $k$, but zeroes
over an algebraic closure $K$.
But first 
we need some more notation and terminology.

The exponents $m$ of monomials lie in $\N^n\subset \R^n_+$. On $\R^n$
we take  coordinates $r=(r_1,\dots,r_n)$ . Let $w=(w_1,\dots,w_n)$ be a
system of \textit{positive} (rational) weights on the variables $x_i$.
We consider $w$ as element in the dual space $(\R^n)^*$. So it defines
a linear function $\lambda\colon r\mapsto \langle w,r\rangle$ on $\R^n$, and a valuation
on $K[[x_1,\dots,x_n]]$ by $w(f)=\min\{\langle w,m\rangle \mid
m\in \Supp(f)\}$. For a subset $\cA\subset \R^n_+$ we set 
$w(\cA)=\min\{\langle w,r\rangle \mid r\in \cA\}$. 
The \textit{initial set} $\ini_w(\cA)$ of $\cA$ is the set $\{r\in \cA \mid 
\langle w,r\rangle=w(\cA)\}$; for a convex 
polytope  it is also called minimising face. For a series $f=\sum a_mx^m \in 
K[[x_1,\dots,x_n]]$ the \textit{initial form} $\ini_w(f)$ is
$f_{\ini_w(\Supp(f))}$, that is $\ini_w(f)=\sum_{m\colon \langle w,m\rangle=w(f)}a_mx^m $.

A (finite)  set of linear functions $\lambda_j$ given by a set of weights
$\{ w^{(j)}\mid j\in J\}$ has a minimum  
$\lambda_J\colon r\mapsto\min_{j\in J} \lambda_j(r)
= \min_{j\in J}\langle w^{(j)},r\rangle$. We suppose the set 
to be irredundant, in that no proper subset has the same minimum. It defines a 
diagram $\Gamma=\{r\in (\R_+)^n \mid \lambda_J(r)=1\}$.
The faces $\Delta_j=\{r\in (\R_+)^n \mid \lambda_j(r)=\lambda_J(r)=1\}$
are non-empty and $(n-1)$-dimensional.
Conversely, given a diagram  $\Gamma$ such that 
the closed region $\Gamma_+$ on and above it is convex  
and central projection 
onto the unit simplex is a bijection,  each facet $\Delta$
defines a unique linear function $\lambda_\Delta$ such that $\lambda_\Delta(r)=1$
for all $r\in \Delta$, that is, 
there is a uniquely defined system of weights  $w_\Delta$ such that all
points $r\in \Delta$ satisfy $\langle w_\Delta,r\rangle=1$. The collection of these linear functions defines a convenient diagram
$\Gamma$ as above. 

\begin{defn}
A convenient and convex diagram $\Gamma$ defined (as above) by
a finite set of positive weights is called a  \textit{$C$-diagram}.
\end{defn}

For an arbitrary subset $I\subset \{1,\dots,n\}$ we denote 
the coordinate subspace $\{(r_1,\dots,r_n)\in \R^n\mid r_i=0 
\text{ if } i\notin I\}$ by $\R^I$. For $K^n$ we use  a similar
notation, so $K^I=\{(x_1,\dots,x_n)\in K^n\mid x_i=0 
\text{ if } i\notin I\}$. Furthermore we put $(K^*)^I= 
(K^*)^n\cap K^I$.
Let $Q=(q_1,\dots,q_n)\in K^n$ be a point. 
We set $I_Q=\{i\mid q_i\neq 0\}$. Then
$\R^{I_Q}=\{(r_1,\dots,r_n)\in \R^n\mid r_i=0 
\text{ if } q_i=0\}$.

Note that Kouchnirenko's  non-degeneracy condition depends only on the principal part of the series $f$.  As the condition 
in Definition \ref{def-nondeg} only involves zeroes on $(K^*)^n$ of the ideal
$(x_1\pabi f_\Delta/1,\dots, x_n\pabi f_\Delta/n )$ for $\Delta$ a closed face of $\Gamma(f)$, one 
can as well require that the ideal
$(\pabi f_\Delta/1,\dots, \pabi f_\Delta/n)$ has no  zero 
on $(K^*)^n$.  
We first reformulate the definition following \cite{Mo-bk}. 
 \begin{defn} 
The  series $f$ is \textit{non-degenerate} if for every  system of positive
weights $w$ the ideal 
\[ 
\left(\pabi \ini_w(f)/1,\dots, \pabi \ini_w(f)/n\right) 
\] 
has no  zero on the torus $(K^*)^n$. 
\end{defn} 

Mondal's non-degeneracy condition 
does not involve the partial derivatives of initial forms, but 
initial forms of the partial derivatives.
\begin{defn}[\cite{Mo-bk}] 
A series $f$ with $f(0)=0$ is \textit{partially non-degenerate} if 
for every non-empty subset $I$ of $\{1,\dots,n \}$
and any  system of positive
weights $w$ on the $x_i$ with $i\in I$ the ideal 
\[ 
\left(\ini_w\Big(\pabi f/1\Big|_{K^I}\Big),\dots, \ini_w\Big(\pabi f/n\Big|_{K^I}\Big) \right) 
\] 
has no  zero on the torus $(K^*)^n$. 
\end{defn} 

This condition involves also terms of the series $f$ different from the
principal part. As example, consider functions $f(x,y)$ with principal
part $f_\Gamma=x^a+y^b$, where $\gcd(a,b)=1$. Then $\ini_w(\pab f_\Gamma/x)
=a x^{a-1}$ for all $w$,
but in general $\ini_w(\pab f/x)$ contains terms involving the variable $y$.

Wall's non-degeneracy condition  is stronger than Kouchnirenko's,
but will be required for less faces.  It starts 
 from a  $C$-diagram $\Gamma$, for which 
the intersection points with the 
coordinate axes need not be lattice points.

\begin{defn} 
A face $\Delta$  is an 
\textit{inner face} of a $C$-diagram $\Gamma$ if it is not contained in any coordinate hyperplane.  
\end{defn}

\begin{defn} 
Let $f$ be a series whose support has no points below the
$C$-diagram $\Gamma$.
The series $f$ is \textit{inner non-degenerate}  
with respect to  $\Gamma$ if for every \textit{inner} 
face $\Delta$ the following holds: 
$\Delta\cap \R^{I_Q}=\emptyset$ 
for each common zero $Q$ of the ideal $(\pabi f_\Delta/1,\dots, \pabi f_\Delta/n)$. 
\end{defn} 

We say that $f$ is inner non-degenerate if there exists a $C$-diagram $\Gamma$
with respect to which $f$ is  inner non-degenerate. Wall calls his condition
NPND$^*$ \cite{Wa}; we follow the terminology of \cite{BGM}, where
the concept is extended to finite characteristic. The condition depends on the
diagram $\Gamma$, and it is not quite clear how it is related to the 
Newton diagram $\Gamma(f)$ of $f$.
The case  $n=2$ is easy to analyse; this is done by Wall \cite{Wa}.
A detailed study of the possible shape of Newton diagrams in $\R^3$
is made by Oleksik \cite{Ol} in connection with the computation of \L ojasiewicz exponents.
He defines an \textit{exceptional face} $\Delta$ of $\Gamma(f)\subset \R^n$ 
as a facet with  one of its vertices  at a distance 1 to a coordinate axis,
while  the remaining vertices define an $(n - 2)$-dimensional face
in one of the coordinate hyperplanes through that axis. A 
combinatorial characterisation of Newton polyhedra 
$\Gamma_+\subset\Gamma'_+$
in $ \R_+^3$  
with $\nu(\Gamma_-)=\nu(\Gamma'_-)$
is given by 
Brzostowski, Krasi\'{n}ski  and  Walewska \cite{BKW}.

If $\dim\Gamma(f)=n-1$ and not convenient one obtains a convenient
diagram by taking the diagram determined by the linear functions
$\lambda_\Delta$ for all facets of $\Gamma(f)$. Here one can leave out 
the exceptional faces. It is not clear from the definition, but we conjecture
that in characteristic zero, if $f$ is inner non-degenerate with respect to
a $C$-diagram $\Gamma$,  and $\Gamma'$ is a 
$C$-diagram with $\Supp(f)\subset \Gamma'_+$ with the same Newton number,
then $f$ is also    inner non-degenerate with respect to $\Gamma'$. 
 
\begin{example}[Example \ref{char13} continued]
The two systems of weights $(\frac 16,\frac14,\frac12)$ and 
$(\frac2{13},\frac3{13},\frac7{13})$ define a $C$-diagram $\Gamma$,
shown in Figure \ref{fig-3}. The function
 $\tilde f=-z^2+2z(y^2-x^3)-4x^5y+x^7$ is inner non-degenerate with respect to 
 $\Gamma$ (if $\Char K$ is not 2, 3 or 13). There are only three inner faces.
 The (reduced) singular set of  $f_\Delta=2z(y^2-x^3)-4x^5y$ 
is the $z$-axis and the face $\Delta$ does not touch this axis.  
\end{example}
 
\begin{figure}
\begin{tikzpicture}[punt/.style={circle,fill=blue,fill opacity=1,inner sep=2pt}]
\draw[help lines] (0,0) grid (7.5,4.7);
\draw[thick, ->] (0,0,0) -- (7.6,0,0) node[right] {$x$}; 
\draw[thick, ->] (0,0,0) -- (0,4.8,0) node[above] {$y$};
\draw[thick, ->] (0,0,0) -- (0,0,2.6) node[below left] {$z$};
\foreach \x in {1,...,7} \draw[shift={(\x,0)}, help lines] node[above right, black] {$\scriptstyle\x$} (0,0) -- (-.96,-.96) ;
\foreach \y in {1,...,4} \draw[shift={(0,\y,0)}, help lines]  node[above right,black] {$\scriptstyle \y$} (0,0) -- (-.96,-.96);
\foreach \z in {1,2} \draw[shift={(0,0,\z)}, help lines]  node[above, black] at (\z * 0.05,0.1,0) {$\ \scriptstyle \z$} (0,0) -- (7.5,0) (0,0) -- (0,4.8);
\draw (0,0,0) node[above right] {$\scriptstyle 0$} (7,0,0) node[punt]{};
\filldraw[blue,thick, fill opacity=.3]
node[punt] at (0,2,1)  {}
(0,2,1) -- (0,0,2) node[punt]{} -- (3,0,1) node[punt]{} -- (6.5,0,0)    
 --  (5,1,0)   node[punt]{} -- (0,4.33,0) -- (0,2,1) -- (3,0,1) ;
\end{tikzpicture}
\caption{A $C$-diagram for $\tilde f=-z^2+2z(y^2-x^3)-4x^5y+x^7$.}
\label{fig-3}
\end{figure}

\begin{example}[Example \ref{muisnu} continued]
Let  $yf(x,y)$ be  a homogeneous isolated curve singularity  of degree $d$ and
consider  the degenerate  function 
$g=(z+x)^d+yf(x,y)$. This function is inner non-degenerate with respect to
$\Gamma$ consisting of the triangle given by the weights $(\frac 1d, 
\frac 1d,\frac 1d)$. The triangle itself is the only inner face, and as $g$ has an
isolated singularity, the non-degeneracy condition is satisfied.

The function is also partially non-degenerate. 
Restricted to $y=0$ and
with weights $(1,1)$, or what
amounts to the same, weights $(\frac 1d,\frac 1d)$  the ideal of initial forms of 
the partial derivatives is generated by $(z+x)^{d-1}$ and $f(x,0)$. As 
$f(x,0)$ is a non-zero multiple of $x^{d-1}$, there are no zeroes
on $(K^*)^3$.
\end{example}

 
\subsection{Relations between the different conditions}
Kouchnirenko's non-degeneracy of a series $f$ does not imply that $f$ is inner or partial
non-degenerate, as non-isolated singularities also can be non-degenerate.
On the other hand, inner non-degenerate functions have 
finite Milnor number \cite{Wa, BGM}, and the same is true for
partially non-degenerate functions.
 
 \begin{prop}
If $f$ is partially non-degenerate, then the origin is an isolated
critical point of $f$, that is $\mu(f)<\infty$.
\end{prop}

\begin{proof}
Suppose that 
$\mu(f)=\infty$. Let $B$ be a branch of a curve contained in the zero set
$V(\pabi f_\Delta/1,\dots, \pabi f_\Delta/n)$.  
Let $I_B=\{i\mid x_i|_B\not\equiv 0\}$. The weights $w$ of an appropriate weighted
tangent cone to $B\subset K^{I_B}$ lead to initial forms violating the
non-degeneracy condition, cf.~\cite[Lemma X.17]{Mo-bk}.
\end{proof}

We have the following relations between the different
non-degeneracy conditions.
 
 \begin{prop}[{\cite[Proposition XII.6]{Mo-bk}}]
 If $f$ is non-degenerate and $\mu(f)<\infty$, then 
$f$ is partially non-degenerate.
\end{prop}
 
\begin{prop}[{\cite[Proposition XII.9]{Mo-bk}}]\label{innerispart}
An inner non-degenerate series is partially non-degenerate.
\end{prop}

We give here the proof of the easiest case, that $f$ is non-degenerate and convenient.
Let  $I\subset\{1,\dots,n \}$
and  $w$ a  system of positive (integral)
weights on the $x_i$ with $i\in I$ be given.
As $f$ is convenient, $\Gamma(f)\cap \R^I\neq\emptyset$.
We can extend $w$ to a system $w'$ of positive  (rational) weights on all $x_i$
such that $\ini_{w'}(\Gamma(f))\subset \R^I$. Then 
$\ini_{w'}(f)$ depends only  on the $x_i$ with $i\in I$. By non-degeneracy
the polynomials $\pabi \ini_{w'}(f)/i$, $i\in I$, have no common zero in
$(K^*)^n$. If for $i\in I$ the polynomial $\pabi \ini_{w'}(f)/i$
is not identically zero, then $\pabi \ini_{w'}(f)/i=\pabi \ini_{w}(f|_{K^I})/i=
\ini_w ( \pabi f/i|_{K^I})$. Those functions do not have a common zero,
so also not all $\ini_w ( \pabi f/i|_{K^I})$ with $i\in \{1,\dots,n \}$.

The other direction of the implication in Proposition \ref{innerispart} 
is not true in finite characteristic (the simplest example is $x^p+x^q$),
but in characteristic zero no counterexamples are
known. It is easy to see that for $n=2$ partial non-degeneracy implies
inner non-degeneracy, and Mondal gives a proof for $n=3$
\cite[XII.30]{Mo-bk}.

\begin{conj}
A partially non-degenerate series $f\in K[[x_1,\dots,x_n]]$, with $\Char K =0$,
is also inner non-degenerate.
\end{conj}

\subsection{Minimal Milnor number}
For series $g_1,\dots,g_n\in  K[[x_1,\dots,x_n]]$ 
the intersection multiplicity (at the origin) is 
\[
[g_1,\dots,g_n]_0=\dim_K K[[x_1,\dots,x_n]]/(g_1,\dots,g_n)\;.
\]
For a collection $(\Gamma_1,\dots,\Gamma_n)$ of $n$ diagrams in 
$\R^n$ define $[\Gamma_1,\dots,\Gamma_n]_0$ as the minimal intersection
multiplicity at the origin of series $g_1,\dots,g_n$ with 
the Newton diagram of $g_i$ on or above $\Gamma_i$. 
Given a subset $\cA \subset \N^n$ define $\partial_i\cA$ as the support of $\pabi f/i$
for any $f\in K[[x_1,\dots,x_n]]$ with $\Supp(f)=\cA$, that is 
\[
\partial_i\cA=\{m-e_j\mid m\in \cA, m-e_j\in \N^n, \text{ $p$ does not divide
           $m_j$}\}\;.
\]
We now state Mondal's main result on the generic Milnor number.
 \begin{theorem}[{\cite[Theorem XII.3]{Mo-bk}}]\label{mondalthm}
Suppose that the minimal intersection multiplicity $[\Gamma(\partial_1\cA),\dots,\Gamma(\partial_n\cA)]_0$ is finite.
For a series $f\in K[[x_1,\dots,x_n]]$ with support in $\cA$ and with
$\Gamma(\pabi f/j)=\Gamma(\partial_j\cA)$ for all $j$ one has
\[\mu(f)\geq [\Gamma(\partial_1\cA),\dots,\Gamma(\partial_n\cA)]_0\]
with equality
if and only if $f$ is partially non-degenerate.
If $\Char K=0$, then series realising equality exist.
\end{theorem}

In characteristic zero the minimal value $[\Gamma(\partial_1\cA),\dots,\Gamma(\partial_n\cA)]_0$  for the Milnor number 
is in fact equal to $\nu(\cA)$. This follows
from 
Kouchnirenko's result  mentioned after Theorem \ref{kouchthm}, that 
there exists a  non-degenerate function $f$ with support equal to $\cA$;
its Milnor number is $\nu(f)$.
It follows from Proposition \ref{innerispart} and Theorem \ref{mondalthm}
that in characteristic zero an inner non-degenerate function satisfies
$\mu(f)=\nu(f)$; the proof of \cite[Theorem 1.6]{Wa} is incomplete,
as it only shows that a non-degenerate, not convenient function $f$ is
right equivalent to a convenient function $f+\sum x_i^m$ with $m\gg0$
(this is \cite[Th\'eor\`eme 3.7 (i)]{Kou}
but Kouchnirenko does not show it in detail), but does not prove
$\mu(f)=\nu(f)$ for convenient degenerate inner non-degenerate functions
(as in Example \ref{muisnu}). Wall's argument does show that
$\nu(f)=\nu(\Gamma_-)$, if $f$ is inner non-degenerate with respect to 
the $C$-diagram $\Gamma$.

\begin{cor}
A series $f\in K[[x_1,\dots,x_n]]$ with $\mu(f)<\infty$ is (inner, partially) degenerate if there exists a series $g$ with $\Supp(g)=\Supp(f)$ and
lower Milnor number: $\mu(g)<\mu(f)$.
\end{cor}

This is easy to check, without determining the faces 
of the Newton diagram. One has  to compute, say with {\sc Singular} \cite{GPS},
$\mu(f)$
and $\mu(g)$ for a general enough function with the 
same support. Taking all 
coefficients equal to 1 might not be general enough; in my experience 
a good choice is to use the coefficients $1,2,3,\dots, k$, if there are $k$ monomials.

 
\section{Finite characteristic} 
\subsection{Weakly non-degenerate functions}
In finite characteristic it is no longer true that the 
Milnor number is invariant under contact equivalence. The simplest example is the  
function $f(x)=x^p$ in characteristic $p$ 
with $\mu(f)=\infty$, while  $\mu(g)=p$  
for the contact equivalent function 
$g(x)=(1+x)f(x)=x^p+x^{p+1}$. 
Recall that   $f, g \in K[[x_1,\dots,x_n]]$ are contact equivalent if there is an 
automorphism $\varphi \in \operatorname{Aut}(K[[x_1,\dots,x_n]])$
and a unit $u \in K[[x_1,\dots,x_n]]^*$ such that $f= u \cdot \varphi(g)$,  
see e.g. \cite[p. 62]{BGM}.

Some invariants depend only on the contact class, that is on the zero set 
of $f$. An example is the $\delta$-invariant  (the number of virtual double points) for plane curve
singularities. The question then arises under which conditions
such invariants can be computed from the Newton diagram. 
A face function $f_\Delta$ is quasi-homogeneous, but in finite 
characteristic it can be that $f_\Delta$ does not lie in the
ideal $(\pabi f_\Delta/1,\dots,\pabi f_\Delta/n)$: in characteristic zero one
can use Euler's identity, but  not if the characteristic divides the weighted degree.
The simplest example is again the polynomial $x^p$. This leads to the 
following definitions.

\begin{defn} 
The  series $f$ is \textit{weakly non-degenerate} if for every closed 
face $\Delta\subset \Gamma(f)$ the polynomials 
\[ 
f_\Delta, \pabi f_\Delta/1,\dots,\pabi f_\Delta/n 
\] 
have no common zero on the torus $(K^*)^n$. 
\end{defn}

\begin{defn} 
Let $f$ be a series whose support has no points below the
convenient diagram $\Gamma$.
The series $f$ is \textit{weakly inner non-degenerate}  
with respect to  $\Gamma$ if for every \textit{inner} 
face $\Delta$ the following holds: 
$\Delta\cap \R^{I_Q}=\emptyset$ 
for each common zero $Q$ of the ideal 
$(f_\Delta,\pabi f_\Delta/1,\dots, \pabi f_\Delta/n)$. 
\end{defn}

In fact, it is rather common in the literature on $p$-adic 
zeta-functions to add the function $f_\Delta$ in the definition of 
non-degeneracy, see e.g \cite{KH}. This is also the definition of
Beelen and Pellikaan \cite{PB} in the case $n=2$; they included convenience.
They refer to Kouchnirenko's definition as  non-degeneracy
in the strong sense. The term weakly non-degenerate is from
\cite{BGM}, where the condition is only asked for top-dimensional faces,
as a direct but nowhere used generalisation of the definition in
Beelen and Pellikaan \cite{PB};
as the
weak non-degeneracy condition is automatically satisfied for zero
dimensional faces,  it suffices in the case  $n=2$
to ask the condition for top-dimensional faces. 
Note that the function $f_\Delta$ is also added in Khovanskii's definition
of non-degenerate Laurent polynomials \cite{Kho}, called 
$0$-non-degenerate by Varchenko \cite{Va}.

\begin{example}[{\cite[Remark 3.15]{PB}}]
Consider $f=x^{p+1}+x^{p-1}y+xy^{p-1}+y^{p+1}$, or more
generally a function $f$ of the form $xyf_{p-2}(x,y)+f_{>p}(x,y)$, with
$xyf_{p-2}(x,y)$ a homogeneous polynomial of degree $p$ with $p$ distinct
factors, and $f_{>p}(x,y)$ a series with multiplicity at least $p+1$, making
the function convenient and the Milnor number finite. Then $f$ is 
(partially) degenerate: for $w=(1,1)$ we have $\ini_w(\pab f/x)=
\pab\ini_w(f)/x=\pab (xyf_{p-2})/x$ and similarly for the derivative w.r.t.~$y$.
As $x\pab g_p/x+y\pab g_p/y=0$ for  any  homogeneous 
polynomial $g_p$ of degree $p$ in $\Char p >0$, we get non-trivial solutions.
But $f$ is weakly non-degenerate, and in fact weakly inner 
non-degenerate with respect to the segment joining $(p,0)$ and $(0,p)$.
\end{example}

\begin{example}\label{char13-2}
The function  $\tilde f=-z^2+2z(y^2-x^3)-4x^5y+x^7$ of Example
\ref{char13}  (see Figure  \ref{fig-2}) degenerates in characteristic 13 on the 
facet $\Delta$ with vertices
$(3,0,1)$, $(0,2,1)$ and $(5,1,0)$. Indeed, as 
\[\begin{vmatrix}
3&0&5\\0&2&1\\1&1&0
\end{vmatrix}
=-13\;,\]
the polynomials $x\pab f_\Delta/x$, $y\pab f_\Delta/y$ and $z\pab f_\Delta/z$
are linearly dependent, whatever the coefficients of the monomials are.
The determinant in question occurs in the computation of the Newton number. 
\end{example}

\begin{example}
The polynomial 
\[
f(x,y,z)= x^py+y^pz + z^px
\]
is non-degenerate and inner non-degenerate.
For every value of $m>p$ the function
\[
f_m=  x^py+y^pz + z^px+x^m+y^m+z^m
\]
is still inner non-degenerate, but degenerate, also when $p \nmid m$.
The face $\delta$ with $f_{m,\delta}=x^py+y^m$ is not an inner face.
The ideal of partial derivatives is generated by 
$\pab f_{m,\delta}/y=x^p+my^{m-1}$.
It follows that $f_m$ is weakly non-degenerate, except when $m=kp+1$.
\end{example}

The previous example shows that the characteristic zero proof
of Kouchnirenko and Wall for the equality $\mu(f)=\nu(f)$ does not extend
to finite characteristic, contrary to what  
Boubakri, Greuel and Markwig claim \cite[Proof of Theorem 7]{BGM}.
The proof uses finite determinacy to conclude that $f$ is equivalent to
$f+\sum x_i^m$ for suitable large $m$ and the equality of Milnor and
Newton number for convenient non-degenerate series (Theorem \ref{kouchthm}).
For $n=2$ the argument does work in finite characteristic, as 
$f(x,y)+x^m+y^m$ is non-degenerate for suitable $m$, if $f(x,y)$ 
is non-degenerate.
\subsection{Conjectures}
For non-degenerate plane curve singularities one can compute the $\delta$-invariant
from the Newton diagram.
Over the complex numbers this is described in \cite[13.3.1]{AGV}:
it is the number $\sigma$ of subdiagrammatic monomials $x^m$, meaning that 
$m+(1,1)$ does not belong to the interior of the Newton polyhedron.
More generally, for $n>2$ the number of subdiagrammatic monomials 
gives the geometric genus of the singularity.
An elementary proof in all characteristics in the plane curve case is given
by Beelen and Pellikaan \cite{PB}.

\begin{prop}
If $f\in K[[x,y]]$ is weakly non-degenerate then $\delta(f)$ is 
equal to $\sigma(f)$, the number of subdiagrammatic monomials.
\end{prop}

As also the number $r$ of branches of $f$ is easily computed
from the Newton diagram,  this gives, as observed in \cite{BGM}:

\begin{theorem}
For a weakly non-degenerate $f\in K[[x,y]]$ one has $\nu(f)=2\delta - r +1$.
\end{theorem}

Observe that in general Milnor's formula $\mu=2\delta -r + 1$ 
does not hold in finite characteristic. The difference $\mu-(2\delta -r + 1)$
is the number of wild vanishing cycles
\cite[p. 265]{MW}. In particular, if $f$ is non-degenerate then there are no 
wild vanishing cycles. We conjecture that this holds in any dimension.
Greuel and Duc  write \cite[p. 579]{G-D}:
\lq Although we can compute the number of wild vanishing cycles, it seems hard
to understand them\rq.
The number of wild vanishing cycles was introduced by Deligne \cite{SGA}.
He defines sheaves of vanishing cycles $R^i\Phi_{\bar \eta}(\Z/\ell)$
(with $\ell\neq p$).  He  gets  the total number of vanishing cycles
as the sum of the  number of (ordinary) vanishing cycles and the
number of wild vanishing cycles.
In the equicharacteristic case Deligne proves that the Milnor number 
is equal to the total number  of vanishing cycles \cite[Expos\'e XVI]{SGA}.

\begin{conj}\label{nowild}
If $f\in K[[x_1,\dots,x_n]]$ is non-degenerate, then there are no 
wild vanishing cycles.
If $f$ is weakly non-degenerate, then $\mu(f)-\nu(f)$ is the number
of wild vanishing cycles.
\end{conj}

This conjecture implies 
\begin{conj} 
In characteristic 13 the function $f=(y^2-x^3)^2-4x^5y+x^7$
of Example \ref{char13} is stably degenerate. 
\end{conj} 

Even simpler examples are obtained from the function $f(x)=x^p$ in $\Char p$.
This function is weakly non-degenerate. As it has $\mu(f)=\infty$,
it can not be stably equivalent to an inner or partially non-degenerate function.
Arnold's problem asks for functions with finite multiplicity.
We take a function with the same zero locus: we consider 
$f_q(x)=x^p+x^q$ with $q>p$ and $p$ and $q$ coprime.
In this case $\mu(f_q)=q-1$, and $f_q$ is not (inner)  non-degenerate.
Conjecture \ref{nowild} implies 

\begin{conj} 
The function $f_q(x)=x^p+x^q$,  $\Char K=p$, $q>p$ and 
$\gcd(p,q)=1$ is stably degenerate. 
\end{conj} 

We note that $f_q(x)=x^p+x^q$ is partially non-degenerate. 
This is caused by  the monomial $x^q$ above the Newton diagram.

\section{Characteristic zero}

In this section we give examples of stably non-degenerate and
(conjecturally) degenerate singularities in the case $\Char K =0$. 
 
\subsection{The basic trick} 
Let $f$ be a (degenerate) function of the form $f=g+m\vp^k$, where 
$m$ is any function, but preferably a monomial. Then we can remove the term $m\vp^k$ 
after stabilisation with two new variables:
\begin{lemma}\label{trick} 
The function $f=g+m\vp^k$ is stably equivalent to $-uv+u\vp+mv^k+g$.  
\end{lemma} 
\begin{proof} We stabilise $f$ with the quadratic form $-\tilde u \tilde v$ in two new
variables and compute the effect
of the coordinate transformation
 $\tilde u =u-m\frac{v^k-\vp^k}{v-\vp}$, $\tilde v = v-\vp$:
\[ 
-\left(u-m\frac{v^k-\vp^k}{v-\vp}\right)(v-\vp)+m\vp^k= 
-uv+u\vp+mv^k\;. 
\] 
\end{proof} 
This formula includes the special case $k=1$: one has that  
$g+m\vp$ is stably equivalent to $-uv+u\vp+vm+g$. 
We note also the case  $m=1$ and $k=2$, where 
we have $f=g+\vp^2$. The basic trick gives  $-uv+u\vp+v^2+g$, 
to which we apply the coordinate transformation $v=\bar v+\bar u$, $u= 2\bar u$, yielding 
${\bar v }^2-{\bar u}^2+2 \bar u\vp+g$, so $f$ is also 
stably equivalent to $-{\bar u}^2+2 \bar u\vp+g$; this is the obvious 
way to treat this case. 
\begin{cor} 
Every polynomial is stably equivalent to a polynomial of degree three. 
\end{cor} 
\begin{proof} 
A product $m\vp^k$ with $\deg m = d$, $\deg \vp=e$ 
can be replaced by 
$-uv+u\vp+v^km$ with summands of degrees $2$, $e+1$ and $d+k$. The condition
that each of these degrees is
less than  $d+ke$ is that  $d>1$  or $k>1$ and that $e>1$. A monomial of degree
at least 4 can always be written as a product $m\vp$ with $d,e\geq2$ and therefore be replaced by 
monomials of lower degree (this might not be the most efficient way to reduce the degree).  
\end{proof}

\begin{remark}\label{twoexample} 
If  $f=g+m_1\vp^{k_1}+m_2\vp^{k_2}$, we can apply our basic 
trick twice to get  
\[ 
-u_1v_1-u_2v_2+(u_1+u_2)\vp+m_1v_1^{k_1}+m_2v_2^{k_2}+g 
\] 
after which we make $u_1+u_2$ into a new variable, say by replacing 
$u_2$ by $u_2-u_1$, giving 
\[ 
-u_1v_1+u_1v_2-u_2v_2+u_2\vp+m_1v_1^{k_1}+m_2v_2^{k_2}+g \;.
\] 
This procedure generalises to more terms. 
\end{remark} 
 
\subsection{Luengo's example}
\begin{example}\label{luengo} 
The function
$f=x^9+y(xy^3+z^4)^2+y^{10}$ \cite{lu}
with $\mu(f)=547$  
has non-smooth $\mu$-constant stratum.
It is stably equivalent to the non-degenerate function
\[-uv+u(xy^3+z^4)+yv^2+y^{10}\;.\] 

The stratum Luengo computed is in fact the $\mu^*$-constant stratum
$S_{\mu^*}$ \cite{lu}.
Recall that
$\mu^*$ is the sequence of the Milnor numbers of repeated hyperplane
sections \cite{Teiss}. If we also know that the topological type is constant
in a $\mu$-constant deformation, or the multiplicity, then it follows that
the stratum $S_{\mu^*}$ is the whole $\mu$-constant stratum
$S_{\mu}$. It is known that  for the function $f$ it 
is least an irreducible 
component of $S_{\mu}$ \cite{st}.
The stratum $S_{\mu^*}$  has a quadratic singularity.
When Luengo did his computation
\cite{luc} on  an IBM 370 with a memory of 8256 K, he could not determine 
the decisive polynomial explicitly. Maybe nowadays it is possible, but it is clear that
the result is too big to be of any use. 

Starting from the first order deformation
\[
f+2(a_{60}x^5+a_{51}x^4z+a_{42}x^3z^2+a_{33}x^2z^3)(xy^3+z^4)
\]
the obstruction to lift it to second order (see \cite{st}) is given by
\[
a_{60}a_{33}+a_{51}a_{42}=0\;.
\]
Some $1$-parameter families are easy to describe.
We can even make the  $\mu$-constant
deformation $f+a_{60}x^5(xy^3+z^4)$ 
stably non-degenerate by the transformation $u\mapsto u-a_{60}x^5$, resulting 
in $ -uv+u(xy^3+z^4)+yv^2+a_{60}vx^5+y^{10}$. 

Furthermore we have
\[
f+a_{42}x^3z^2(xy^3+z^4)-a_{42}^2x^7y^2
\]
and
\[
f+a_{33}x^2z^3(xy^3+z^4)-a_{33}^2x^5y^2z^2+a_{33}^3x^7yz \;.
\]
For the deformation in the $a_{51}$-direction we computed up to order 30 in the deformation variable, but we have not been able to find
a $\mu$-constant deformation.

The $a_{33}$-deformation is stably equivalent to
\[
-uv +x^9+y^{10}+ u(xy^3+z^4)+yv^2+2a_{33}vx^2z^3-a_{33}^2x^5y^2z^2 +a_{33}^3x^7yz\;.
\]
By changing the coefficient of $x^5y^2z^2$ the Milnor number drops to 533. 
The polynomial degenerates on the face $\Delta$ which is the intersection
of the facets with normalised weight vectors $(5,4,1,1,1)/9$
 and $(128,99,26,22,23)/220$ respectively
and $f_\Delta=u(xy^3+z^4)+yv^2+2a_{33}vx^2z^3-a_{33}^2x^5y^2z^2$.
We can write this expression as symmetric determinant:
\[
f_\Delta=
\begin{vmatrix}
-u & -v & a_{33}x^2z \\
-v & xy^2 & -z^2 \\
a_{33}x^2z  & -z^2 & -y
\end{vmatrix}\;.
\]
This hypersurface is singular on the codimension 3 space defined by the
$2\times2$-minors of the above matrix. It is reducible, with one component
in $z=y=v=0$ and the other having as normalisation  the cone over
the rational normal curve of degree 4; it can be parametrised as
$z=-st^3$, $y=-y^4$, $x=s^4$, $v=-a_{33}s^{11}t^5$ and 
$u=-a_{33}^2s^{18}q^2$. Note also that the weight of the monomial 
$x^7yz$ is larger than 1 for both weight vectors.
We conjecture that already this polynomial provides an example
of a function which is  stably degenerate.
\end{example} 
  
More generally, we make the
following conjecture.

\begin{conj}\label{mu-conj}
A general function on a non-smooth  $\mu$-constant stratum
is stably degenerate.
\end{conj}

In particular this would prove
\begin{conj}
There exist a  stably degenerate function, which is a
$\mu$-constant deformation of a 
Newton non-degenerate function.
\end{conj}
The existence of such a function
implies that the question of stable non-degeneracy cannot be decided
with invariants depending only on the embedded topological type.

Besides the fact that I do not know what to do in the example above,
the heuristic for conjecture \ref{mu-conj} is the following.
The $\mu$-constant stratum
has very complicated equations and in fact it is not known in a single case
how to write down a general function on the stratum, whereas the
Newton diagram seems to be a relatively simple, combinatorial object.
Furthermore, the non-degenerate functions with the same Newton diagram 
as a generic function
on the stratum should dominate the $\mu$-constant stratum and this fits
badly with the non-smoothness. However this idea does not lead to a proof,
as the coordinate transformations involved need not extend to the
original function (as for the function  $g$ in Example \ref{adp} below).

\subsection{du Plessis' examples}
In \cite{dp} Andrew du Plessis gave 
in a systematic way examples of hypersurfaces of degree $d$ in 
$\P^n$, whose
singularities are not versally deformed by the family $H_d(n)$ of all hypersurfaces
of degree $d$ in $\P^n$. Then at the corresponding point the stratum of hypersurfaces with exactly these 
singularities can be smooth of dimension larger than the expected dimension or it can be singular.
A classical example of the first case is Segre's family of curves of degree $6k$ 
of the form $(f_{3m})^2+(f_{2m})^3$ with $6k^2$ cusps, see  \cite[VIII.5]{Zar}.

If the stratum is singular, we obtain
by adding a suitable form of degree $d+1$ to the equation of the hypersurface
a superisolated singularity with non-smooth $\mu^*$-constant stratum
(and probably also  non-smooth $\mu$-constant stratum, but this has to
be proved, as in \cite{st} for the case of Luengo's example). As 
communicated by du Plessis, the smallest example constructed this way
is the following, with 
$d=3$ and $n=7$. 

\begin{example}\label{adp}
Consider the following function, cf. \cite[Examples 2.7]{dp}:
\[
f=f_3+x_0^4=  x_0x_1^2+x_2(x_2^2+x_3^2-x_0^2)+x_4(x_4^2-x_5^2+x_0^2)
  +x_6^3+x_7^3+x_0^4
\]
with $\mu= 272$; the  projective  hypersurface  $f_3=0$
has four $D_4$-singularities. The hyperplane section $\{x_1=0\}$
belongs to a stratum in $H_3(6)$ with larger than expected dimension.

The $\mu^*$-constant stratum of $f$ is singular with quadratic singularity:
the obstruction to lift the first order deformation
\[
f_3+2x_1(a_{67}x_6x_7+a_{57}x_5x_7+a_{56}x_5x_6+a_{57}x_3x_6+a_{37}x_3x_7+a_{35}x_3x_5)
\]
is $a_{67}a_{35}+a_{57}a_{36}+a_{56}a_{37}=0$. In the chart $x_0=1$ one completes 
the square $\big(x_1+(a_{67}x_6x_7+a_{57}x_5x_7+a_{56}x_5x_6+a_{57}x_3x_6+a_{37}x_3x_7+a_{35}x_3x_5)\big)^2$ 
and the obstruction to find an equivalent polynomial of degree three
is the coefficient of $x_3x_5x_6x_7$.
Also here some 1-parameter $\mu$-constant deformations are easy to write down:
\begin{align*}
f&+2a_{67}x_1x_6x_7\\
f&+2a_{36}x_1x_3x_6+a_{36}^2x_6^2\\
f&+2a_{35}x_1x_3x_5+a_{35}^2x_0(x_3^2+x_5^2-x_0^2)
\end{align*}
and similar ones obtained by symmetry,
but a general deformation is not known explicitly.

The $a_{67 }$-deformation has non-degenerate Newton diagram.
We now show that for a fixed value of $a_{35}$ the function
is equivalent to an inner non-degenerate function.
As $x_6$ and $x_7$ do not occur 
in the $a_{35 }$-deformation we might as well leave them out.
Consider therefore the polynomial $g$ given by
\[
x_0x_1^2+x_2(x_2^2+x_3^2-x_0^2)+x_4(x_4^2-x_5^2+x_0^2)+x_0^4
+2a_{35}x_1x_3x_5+a_{35}^2x_0(x_3^2+x_5^2-x_0^2)
\]
with $\mu(g)=68$.
It degenerates on $\Delta$ with 
$g_\Delta=x_0x_1^2+2a_{35}x_1x_3x_5+a_{35}^2x_0(x_3^2+x_5^2-x_0^2)$.
We  write $g_\Delta$ as determinant:
\[
g_\Delta=
-a_{35}^2
\begin{vmatrix}
x_0 & x_3& x_5 \\
x_3 & x_0 & -x_1/a_{35} \\
x_5  & -x_1/a_{35} & x_0
\end{vmatrix}\;.
\]
In this case the singular locus is reducible and consists of four linear 
spaces, which we can move into coordinate subspaces by
a coordinate transformation:
\begin{align*}
x_0&\mapsto x_0+x_3+x_5-x_1/a_{35}\\
x_1&\mapsto -a_{35}x_0+a_{35}x_3+a_{35}x_5+x_1\\
x_3&\mapsto x_0+x_3-x_5+x_1/a_{35}\\
x_5&\mapsto x_0-x_3+x_5+x_1/a_{35}
\end{align*}
It transforms the original function into
\begin{multline*}
16a_{35}(x_0x_1x_3+x_0x_1x_5+x_1x_3x_5-a_{35}x_0x_3x_5)\\
+x_2^3-4x_2(x_0+x_3)(x_5-x_1/a_{35})+x_4^3+4x_4(x_0+x_5)(x_3-x_1/a_{35})\\
+(x_0+x_3+x_5-x_1/a_{35})^4\;.
\end{multline*}
Now $\mu=\nu=68$. The polynomial still degenerates on some faces
in coordinate hyperplanes. The Newton diagram has 803
compact faces (computed with {\sc G\'ermenes} \cite{Ger}).
It can be checked that the polynomial is inner non-degenerate.
Presumably it can be made non-degenerate with the basic trick.
Note that the coordinate transformation only works for $a_ {35}\neq0$.
\end{example}

%

\subsection{Small examples}
Motivated by the above examples we search for a simpler example,
with low Milnor number. We start from a quasi-homogeneous singularity
$f_\delta$
with one-dimensional singular locus, which is 
generically reduced (to get a small example);
denote by  $\Sigma$ the reduced singular locus.
There should be no coordinate transformation
which  moves $\Sigma$ into a coordinate hyperplane.
Therefore $\Sigma$ should be irreducible.
Moreover, it should not be a complete intersection, where
our methods apply, see Example \ref{ex40}.

The condition that $f_\delta$ is singular along $\Sigma=V(I)$ is that
$f_\delta \in \int I$, where 
$\int I$ is the \textit{primitive ideal} \cite{se, pe}:
\[ 
\textstyle \int  I = \big\{ g\in K[[x_1,\dots,x_n]] \mid 
          (\pabi g/1,\dots, \pabi g/n)\subset I\big\}\;. 
\] 
The terminology is from Pellikaan \cite{pe}. 

Interesting examples can be found in the work of De Jong and Van Straten 
on rational quadruple points \cite[Proposition 1.8]{td}. 
The easiest example is the following. 

\begin{example}\label{ex23}
Let $\Sigma=V(I)$ be the monomial  
curve $(t^3,t^4,t^5)$. Its ideal is given by the minors of  
a $2\times3$ matrix. 
We define a function $f_\delta$ with $\Sigma$ as singular locus
by adding one row to the matrix to get 
the following symmetric $3\times 3$ determinant: 
\[ 
f_\delta=- 
\begin{vmatrix} 
x & y & z \\ 
y& z & x^2 \\ 
z & x^2 & xy 
\end{vmatrix}\;. 
\] 
All $2\times2$ minors lie in the ideal $I$, because they vanish on the curve
$(t^3,t^4,t^5)$,
so by the product rule the partial derivatives of $f_\delta$ lie also 
in $I$, showing that $f_\delta\in \int I$. 
 
We find isolated singularities in three related series by adding suitable 
monomials: 
\begin{alignat*}{3} 
f_{7+3k}&=f_\delta+x^{k}&&=x^5+xy^3+z^3-3x^2yz&&+ x^{k}\;,\\ 
f_{8+3k}&=f_\delta+x^{k-1}y&&=x^5+xy^3+z^3-3x^2yz&&+ x^{k-1}y\;,\\ 
f_{9+3k}&=f_\delta+x^{k-1}z&&=x^5+xy^3+z^3-3x^2yz&&+ x^{k-1}z\;. 
\end{alignat*} 
The lower index denotes the Milnor number. We can write it as $\mu=7+v$, 
where $v$ denotes the weight of the added monomial  (using the 
weights $3,4,5$). The smallest example is $f_{23}$.
The functions $f_\mu$ degenerate by construction on the face 
$\delta$ of the Newton diagram with vertices $(5,0,0)$, $(1,3,0)$ and $(0,0,3)$.

To determine the resolution graph we look at $f_\delta$ in the chart 
$x=1$. It is given by $1+y^3+z^3-3yz=0$, on which we have the 
$\mathbb{Z}_3$-action $(1,y,z)\mapsto(1,\ep y,\ep^2 z)$. We find that
the singularity $f_{7+v}$ has  the same resolution graph  
as the maximal elliptic singularity $z^2+y^3+y^2x^8+x^{9+v}$.
It has $Z^2=-1$, there is a cycle of $v-15$ rational curves, all but one 
having self intersection $-2$, and at the only $(-3)$-curve a chain of three 
$(-2)$-curves is attached.  One has $p_g(f_{7+v})=2$, whereas the 
maximal elliptic singularity has $p_g=4$.
\end{example}

\begin{conj} 
For every function $\tilde f$, stably equivalent to the function 
$f_{7+v}$ of Example \ref{ex23} with $v>15$, one has
$\mu(f)=\mu(\tilde f)>\nu(\tilde f)$. In particular, 
the function $f_{7+v}$  is stably degenerate. 
\end{conj}

If we change the coefficients in  the matrix defining
$f_\delta$ the function will define a non-degenerate function. Every
transformation I tried can also be done for the non-degenerate function,
leading to the same Newton diagram. These transformations involve 
somehow the generators of the ideal $I$, but by changing their coefficients
the ideal will become a complete intersection.
This does however not exclude 
the existence of a very strange coordinate transformation, which 
does the trick.

\begin{example}\label{ex40}
If we take only two generators of the ideal $I$ of the monomial curve
$(t^3,t^4,t^5)$ we get as reduced singular locus the union of this  
curve and the $z$-axis; it is the complete intersection
$x^3-yz=xz-y^2=0$.
We take the function 
\[
g_\delta=(x^3-yz)(xz-y^2)\;.
\]
We have $\mu(g_\delta+z^k)=6k+16$ for $k\geq 4$.
The function $g_\delta+z^k$ is stably equivalent to the non-degenerate function
\[
-uv +u(xz-y^2) + v(x^3-yz)+z^k\;.
\]
\end{example}

\begin{example}\label{ex48}
If we take $h_\delta \in I^2$, with still the same ideal $I$, we can
apply the basic trick as first step. 
Take the function 
\[
h_\delta=(x^3-yz)^2+(xz-y^2)(yx^2-z^2)\;.
\]
The reduced singular locus consists of the monomial  
curve $(t^3,t^4,t^5)$ and the $y$-axis, and is not a complete intersection.
We have $\mu(h_\delta+y^k)=23+5k$ for $k\geq5$, but $\nu(h_\delta+y^k)=41+k$.

We apply the basic trick to $h=h_\delta+y^k$ and get
\[
\tilde h= -uv-w^2+u(xz-y^2)+v(yx^2-z^2)+2w(x^3-yz)+y^k\;.
\]
This polynomial degenerates on the four-dimensional face $\delta'$ where
$\tilde h_{\delta'}=u(xz-y^2)+v(yx^2-z^2)+2w(x^3-yz)$.
Indeed, the function $h_\delta$ involves all five monomials of the given degree
and a general  non-degenerate function can be written as 
$H_\delta=(ax^3-byz)^2+(cxz-dy^2)(eyx^2-fz^2)$. Therefore the same
type of transformation can be applied to $H_\delta+y^k$, leading to the same Newton diagram.
 
As we have a relation $\sum r_if_i$ between the generators $f_i$ of $I$
we get by deriving a relation between the partial derivatives of the $f_i$,
holding modulo $I$. This can be written in terms of  the parameter $t$
in the parametrisation $(t^3,t^4,t^5)$ of $\Sigma$. If $(u,v,2w)$ is 
a multiple of the relation vector we get a non-trivial solution.
Note that we can write $\tilde h_{\delta'}$ as determinant:
\[
\tilde h_{\delta'}=
\begin{vmatrix} 
x & y & z \\ 
y& z & x^2 \\ 
v & -2w & u 
\end{vmatrix}\;. 
\]
This shows that the singular locus of $\tilde h_{\delta'}$ is not contained in the 
coordinate hyperplanes.

As there seems no way to use the fact that there are relations between 
the generators $f_i$ of $I$, we conjecture that also $h_\delta+y^k$ is  stably 
degenerate. 
\end{example}

It does not suffice that $\Sigma$ is not a complete intersection, as the 
following example shows.
 
\begin{example}\label{exkkk}
Consider
\[ 
f_\delta=- 
\begin{vmatrix} 
x & y & z \\ 
y& z & x \\ 
z & x & y 
\end{vmatrix}\;. 
\] 
Then $f_\delta=x^3+y^3+z^3-3xyz$, which is a product of three linear 
factors,  and $f_\delta+l^k$ for a 
general linear function $l$ (a coordinate function will do) is equivalent to
a function of type $T_{k,k,k}= x_1^k+x_2^k+x_3^k+a x_1x_2x_3$
\cite[15.1]{AGV-I}, so a simple coordinate transformation makes the function
non-degenerate. But it is even possible to make it 
non-degenerate after stabilisation, keeping the original 
$(x,y,z)$-coordinates.
The equation has the form 
$l_1l_2l_3+l^k$. 
We apply the basic trick, first once: 
\[ 
-u_1v_1+u_1l_1+v_1l_2l_3 +l^k
\] 
and then once again: 
\[ 
-u_1v_1-u_2v_2+u_1l_1+u_2l_2+v_1v_2l_3+l^k\;. 
\] 
\end{example}
 
\begin{example}\label{ex103}
We give a non-trivial example with $\Sigma$ a complete intersection.
Consider the curve with parametrisation $(t^4,t^5,t^6)$ and equations
$y^2-xz=z^2-x^3=0$. 
This is  the simple complete intersection
curve $W_8$ in Giusti's notation, see \cite[9.8]{AGV-I}.
We take as $f_\delta$ a rather general element in the 
square of the ideal.
Let
\[
f=x(z^2-x^3)^2-z(z^2-x^3)(y^2-zx)+x^2(y^2-zx)^2+xy^5
\]
with $\mu(f)=103$.
The (reduced) singular locus of 
$f_\delta=x^7+x^2y^4-2x^4z^2+2xz^4-y^2z^3-x^3y^2z$ 
consists of $W_8$ and the
$y$-axis. We apply the basic trick several times and simplify.
The result can be seen directly:
a function of the form 
$g=\alpha \vp^2 + \beta \vp\psi + \gamma \psi^2$ is
stably equivalent to
$-uv-tw+u\vp+t\psi+\alpha v^2 + \beta vw + \gamma w^2$, as
the last expression is equal to
\[g
-(u-\beta w-\alpha (v+\vp))(v-\vp)-(w-\psi)(t-\gamma(w+\psi)-\beta \vp)\;.\]
Therefore $f$ is stably equivalent to 
\[
\tilde f=-uv-tw+u(z^2-x^3)+t(y^2-zx)+x v^2 - z vw + x^2 w^2+xy^5\;.
\]
\end{example}

The examples above show that there is no easy criterion for a function
to be stably equivalent to a non-degenerate function.
Our strategy is to remove $f_\Delta$, if  $f$ degenerates on $\Delta$.
In order to increase the Newton number in this way the function
$f_\Delta$ should have a specific form, which is different from
the generic function with the same support, as in Examples \ref{luengo},
\ref{exkkk} and \ref{ex103}.
In Examples \ref{ex23} and \ref{ex48} the only specific structure
is the existence of relations between the generators of the ideal
of the singular locus, but that does not seem to help.
 
\section{Irreducible plane curve singularities} 
It is well known that the only non-degenerate irreducible
plane curve singularities are those with one characteristic pair
($g=1$ in the notation below).
This follows from Newton's method to find a Newton-Puiseux 
series, see e.g. \cite[8.3]{B-K}; indeed the Newton polygon
was introduced by Newton for this purpose.
In this section we give evidence that all irreducible plane curve 
singularities are stably non-degenerate (in characteristic zero).

We describe equations for irreducible plane curve singularities 
following Teissier \cite{Te}, see also  \cite{C-N}.  
We look at algebroid curves over an algebraically closed field 
$K$ of characteristic zero. The basic invariant is the semi-group.

Let $S=\langle\bar\beta_0, \dots,\bar\beta_g\rangle$ be the 
semigroup of the curve. Define numbers $n_i$ by  
$e_i=\gcd(\bar\beta_0, \dots,\bar\beta_i)$ 
and $e_{i-1}=n_ie_i$. 
The condition that $S$ comes from a plane curve  
singularity, is that $n_i\bar \beta_i\in  
\langle\bar\beta_0, \dots,\bar\beta_{i-1}\rangle$ 
and $n_i\bar \beta_i < \bar \beta_{i+1}$.


Teissier showed that every plane curve singularity 
with semigroup $S$ occurs in the positive weight 
part of versal  deformation of the monomial curve $C_S$ 
with the same  semigroup $S$. Embed $C_S$ 
in $K^{g+1}$ by $u_i=t^{\bar \beta_i}$. 
Write  
\[ 
n_i\bar\beta_i= l_0^{(i)}\bar\beta_0+ l_1^{(i)}\bar\beta_1 
+ \dots +  l_{i-1}^{(i)}\bar\beta_{i-1} 
\;. 
\] 
The curve  $C_S$ is  a complete intersection with  
equations  
\begin{align*} 
f_1&=u_1^{n_1}-u_0^{ l_0^{(1)}}=0\\ 
f_2&=u_2^{n_2}-u_0^{ l_0^{(2)}}u_1^{ l_1^{(2)}}=0\\ 
&\;\;\vdots\\ 
f_g&=u_g^{n_g}-u_0^{ l_0^{(g)}}\dots u_{g-1}^{ l_{g-1}^{(g)}}=0 
\end{align*} 
A particular simple deformation of positive weight is given by 
$f_i+\ep u_{i+1}$, $i<g$, and we may even take $\ep=1$. It is then 
easy to eliminate the $u_i$ with $i\geq 2$ to obtain an  
equation of a plane curve. Cassou-Nogu\`es \cite{C-N} has shown that one 
can write the whole equisingular deformation of this 
particular curve as $\wt f_i+ u_{i+1}$, where $\wt f_i$ 
only depends on the coordinates $u_0,\dots,u_i$, so it 
is possible to do the same elimination for the whole stratum. 
However, as  the curve is no longer quasi-homogeneous 
it is not clear whether every plane curve occurs in this 
family. 
 
The easiest elimination occurs when $l_j^{(i)}=0$ for all 
$j\geq 2$ and all $i$. Such semigroups exist for all $g$. They 
can be constructed inductively. Given $\langle\bar\beta_0, \dots,\bar\beta_{g-1}\rangle$ with 
$\gcd(\bar\beta_0, \dots,\bar\beta_{g-1})=1$ and such that $l_j^{(i)}=0$ for $j\geq 2$, 
take a semigroup $\langle n_g\bar\beta_0, \dots,n_g\bar\beta_{g-1},\bar\beta_g\rangle$ with 
$\gcd(n_g,\bar\beta_g)=1$,   $\bar\beta_g> n_{g-1}n_g\bar\beta_{g-1}$ 
and $\bar\beta_g\in \langle \bar\beta_0, \bar\beta_1\rangle$. 
 
\begin{conj} 
The deformed curve $f_i+u_{i+1}$, with  $l_j^{(i)}=0$ for all 
$j\geq 2$, is stably equivalent to a  non-degenerate singularity. 
\end{conj} 
 
\begin{proof}[``Proof"] 
In this case the equation of the plane curve is 
\[ 
\left(\dots\Big((u_1^{n_1}-u_0^{ l_0^{(1)}})^{n_2}-u_0^{ l_0^{(2)}}u_1^{ l_1^{(2)}}\Big)^{n_3} \dots
-u_0^{ l_0^{(g-1)}}u_1^{ l_1^{(g-1)}} \right)^{n_g}-u_0^{ l_0^{(g)}}u_1^{ l_1^{(g)}}=0 \;.
\] 
This is of the form $\vp_g^{n_g}-u_0^{ l_0^{(g)}}u_1^{ l_1^{(g)}}=0$, and 
$\vp_g=\vp_{g-1}^{n_{g-1}}-u_0^{ l_0^{(g-1)}}u_1^{ l_1^{(g-1)}}$ is itself of the same form. The 
principal part is a complete $n_g$-th power. 
We apply the basic trick (Lemma \ref{trick}) and write 
\[ 
-v_gw_g + v_g\vp_g+w_g^{n_g} -u_0^{ l_0^{(g)}}u_1^{ l_1^{(g)}}\;. 
\] 
Here $ v_g\vp_g 
=v_g\left(\vp_{g-1}^{n_{g-1}}-u_0^{ l_0^{(g-1)}}u_1^{ l_1^{(g-1)}}\right)$, so  
we apply the basic trick once more, now to $v_g\vp_{g-1}^{n_{g-1}}$, and obtain 
\[ 
-v_gw_g -v_{g-1}w_{g-1} + v_{g-1}\vp_{g-1}+v_gw_{g-1}^{n_{g-1}}+w_g^{n_g}
-v_gu_0^{ l_0^{(g-1)}}u_1^{ l_1^{(g-1)}} -u_0^{ l_0^{(g)}}u_1^{ l_1^{(g)}}\;. 
\] 
The next step takes care of $v_{g-1}\vp_{g-1}$ and we continue 
inductively. The final result is 
\begin{multline*} 
-v_gw_g -\dots - v_{2}w_{2} +v_2(u_1^{n_1}-u_0^{l_0(1)}) 
+v_3w_{2}^{n_{2}}+\dots +w_g^{n_g}\\{}-v_3u_0^{ l_0^{(2)}}u_1^{ l_1^{(2)}}-\dots- 
v_gu_0^{ l_0^{(g-1)}}u_1^{ l_1^{(g-1)}} -u_0^{ l_0^{(g)}}u_1^{ l_1^{(g)}}\;. 
\end{multline*} 

It remains to show that the final function is non-degenerate. We will not do this,
leaving this as conjecture. In fact, 
we conjecture that all facets of the 
Newton diagram are simplices, implying  non-degeneracy. 
We checked this in the case $g=3$. There are eight monomials, $v_3w_3$, $v_2w_2$, $v_2u_1^{n_1}$, 
$v_2u_0^{l_0^{(1)}}$, $v_3w_2^{n_2}$, $w_3^{n_3}$, 
$v_3u_0^{ l_0^{(2)}}u_1^{ l_1^{(2)}}$ and $u_0^{ l_0^{(3)}}u_1^{ l_1^{(3)}}$. The facets containing both 
$v_2u_1^{n_1}$ and $v_2u_0^{l_0(1)}$ are rather easy to describe, but the
remaining 
facets, on which only one of  $v_2u_1^{n_1}$ and $v_2u_0^{l_0(1)}$ lies, are more difficult, as they depend 
on the values of $l_k^{(i)}$. Each such a facet contains  
exactly six points and is therefore a simplex.
\end{proof}

\begin{remark} 
Without the assumption $l_j^{(i)}=0$ for all 
$j\geq 2$ the situation is more complicated and we  only give 
the case $g=4$. 
The equation is now 
\begin{multline*} 
\left(\Big( 
(u_1^{n_1}-u_0^{ l_0^{(1)}})^{n_2}- u_1^{ l_1^{(2)}}u_0^{ l_0^{(2)}} 
\Big)^{n_3} - 
(u_1^{n_1}-u_0^{ l_0^{(1)}})^{l_2^{(3)}}u_1^{ l_1^{(3)}} u_0^{ l_0^{(3)}} 
\right)^{n_4}-\\ \Big( 
(u_1^{n_1}-u_0^{ l_0^{(1)}})^{n_2}- u_1^{ l_1^{(2)}}u_0^{ l_0^{(2)}} 
\Big)^{l_3^{(4)}}(u_1^{n_1}-u_0^{ l_0^{(1)}})^{l_2^{(4)}}u_1^{ l_1^{(4)}}u_0^{ l_0^{(4)}}\;. 
\end{multline*} 
We start with one application of the basic trick (Lemma \ref{trick}) 
to get 
\begin{multline*} 
-v_4w_4+w_4^{n_4}+v_4\Big( 
(u_1^{n_1}-u_0^{ l_0^{(1)}})^{n_2}- u_1^{ l_1^{(2)}}u_0^{ l_0^{(2)}} 
\Big)^{n_3} - 
v_4(u_1^{n_1}-u_0^{ l_0^{(1)}})^{l_2^{(3)}}u_1^{ l_1^{(3)}} u_0^{ l_0^{(3)}} 
-\\ \Big( 
(u_1^{n_1}-u_0^{ l_0^{(1)}})^{n_2}- u_1^{ l_1^{(2)}}u_0^{ l_0^{(2)}} 
\Big)^{l_3^{(4)}}(u_1^{n_1}-u_0^{ l_0^{(1)}})^{l_2^{(4)}}u_1^{ l_1^{(4)}}u_0^{ l_0^{(4)}}\;. 
\end{multline*} 
Let $\vp_3=(u_1^{n_1}-u_0^{ l_0^{(1)}})^{n_2}- u_1^{ l_1^{(2)}}u_0^{ l_0^{(2)}}$. Then we have two terms involving a power of $\vp_3$, 
so we apply the basic trick twice, followed by a coordinate transformation 
as in Remark \ref{twoexample} to get 
\begin{multline*} 
-v_4w_4-v_{3,1}w_{3,1}+v_{3,1}w_{3,2}-v_{3,2}w_{3,2}+w_4^{n_4}+v_4w_{3,1}^{n_3} 
+v_{3,2}(u_1^{n_1}-u_0^{ l_0^{(1)}})^{n_2}- v_{3,2}u_1^{ l_1^{(2)}}u_0^{ l_0^{(2)}} \\ 
-v_4(u_1^{n_1}-u_0^{ l_0^{(1)}})^{l_2^{(3)}}u_1^{ l_1^{(3)}} u_0^{ l_0^{(3)}} 
-w_{3,2}^{l_3^{(4)}}(u_1^{n_1}-u_0^{ l_0^{(1)}})^{l_2^{(4)}}u_1^{ l_1^{(4)}}u_0^{ l_0^{(4)}}\;. 
\end{multline*} 
Finally we introduce six new variables to handle the powers 
of $\vp_2=u_1^{n_1}-u_0^{ l_0^{(1)}}$. 
\begin{multline*} 
-v_4w_4-v_{3,1}w_{3,1}+(v_{3,1}-v_{3,2})w_{3,2} 
-v_{2,1}w_{2,1}-v_{2,2}w_{2,2}+(v_{2,1}+v_{2,2}-v_{2,3})w_{2,3}\\ 
+w_4^{n_4}+v_4w_{3,1}^{n_3}+v_{3,2}w_{2,1}^{n_2} 
+v_{2,3}(u_1^{n_1}-u_0^{ l_0^{(1)}}) 
- v_{3,2}u_1^{ l_1^{(2)}}u_0^{ l_0^{(2)}} \\ 
-v_4w_{2,2}^{l_2^{(3)}}u_1^{ l_1^{(3)}} u_0^{ l_0^{(3)}} 
-w_{3,2}^{l_3^{(4)}}w_{2,3}^{l_2^{(4)}}u_1^{ l_1^{(4)}}u_0^{ l_0^{(4)}}\;. 
\end{multline*} 
\end{remark}


\begin{thebibliography}{99} 
\bibitem[SGA 7]{SGA}
P. Deligne et  N. Katz, \textit{Séminaire de Géométrie Algébrique du Bois Marie - 1967-69 - 
Groupes de monodromie en géométrie algébrique} - (SGA 7 II) - Lect. Notes  Math. \textbf{340}. Berlin; New York: Springer-Verlag (1973).
\doi{10.1007/BFb0060505}


\bibitem{arn} 
V.I. Arnold,  
\textit{Some open problems in the theory of singularities.} In: 
Singularities, Proc. Symp. Pure Math. {\bf40}, Part 1, pp. 57--69 (1983). 

\bibitem{AP}
Vladimir I. Arnold (ed),
\textit{Arnold’s Problems.} 
Berlin: Springer-Verlag, (2005). \\
\doi{10.1007/b138219}

\bibitem{AGV-I}
V. I.  Arnol'd, S. M. Guse\u{\i}n-Zade and A. N. Varchenko,
\textit{Singularities of differentiable maps. Volume I: The classification of critical points, caustics and wave fronts}. 
(Monographs in Mathematics, Vol. \textbf{82})
Birkhäuser Boston, Inc.,   Boston, MA (1985).
\doi{10.1007/978-1-4612-5154-5}  

  

\bibitem{AGV}
V. I.  Arnol'd, S. M. Guse\u{\i}n-Zade and A. N. Varchenko,
\textit{Singularities of differentiable maps. Volume II: Monodromy and asymptotics of integrals}. 
(Monographs in Mathematics, Vol. \textbf{83})
 Boston, MA: Birkh\"{a}user Boston, Inc., (1988).
\doi{10.1007/978-1-4612-3940-6}
		
		

\bibitem{PB}
Peter Beelen and Ruud Pellikaan, 
\textit{The Newton Polygon of Plane Curves with Many Rational Points.}  
Designs, Codes and Cryptography {\bf21} (2000), 41--67. \\
\doi{10.1023/A:1008323208670}
 

 
\bibitem{BGM}  
Yousra Boubakri, Gert-Martin Greuel and Thomas Markwig,  
\textit{Invariants of hypersurface singularities in positive characteristic.} 
Rev. Mat. Complut. {\bf25} (2012), 61--85. 
 \doi{10.1007/s13163-010-0056-1}
 
\bibitem{B-K} 
Egbert Brieskorn and Horst Kn\"orrer,
\textit{Ebene algebraische Kurven.} 
Birkhäuser, Basel, Boston, Stuttgart  (1981).

 
\bibitem{BKW}
Szymon Brzostowski,  Tadeusz Krasi\'{n}ski  and Justyna Walewska,
\textit {Arnold's problem on monotonicity of the Newton number for
              surface singularities.}
J. Math. Soc. Japan, \textbf{71} (2019), 1257--1268.
\doi{10.2969/jmsj/78557855} 
 
 
\bibitem{C-N} 
P.  Cassou-Nogu\`es,  
\textit{Courbes de semi-groupe donn\' e.}  
Rev. Mat. Univ. Complutense Madr. {\bf4} (1991), 13--44. 
\doi{10.5209/rev_REMA.1991.v4.n1.17992}


 
\bibitem{GPS} 
W. Decker,  G.-M. Greuel, G. Pfister and H. Sch\"one\-mann. 
\href{http://www.singular.uni-kl.de}{{\sc Singular} {4-1-2}} --- {A} computer algebra system for polynomial computations.  (2019). 



\bibitem{KH}
Jan Denef and Kathleen Hoornaert,
\textit{Newton Polyhedra and Igusa's Local Zeta Function}.
Journal of Number Theory 89, 31--64 (2001)
\doi{10.1006/jnth.2000.2606} 

\bibitem{G-D}
Gert-Martin Greuel and Hong Duc Nguyen,
\textit{Some remarks on the planar Kouchnirenko's theorem}. 
Rev. Mat. Complut. {\bf25} (2012), 557--579. 
\doi{10.1007/s13163-011-0082-7}


 
\bibitem{td} 
Theo de Jong and Duco van Straten, 
\textit{On the base space of a semi-universal deformation of rational quadruple 
points.}  
Ann. of Math. (2) {\bf134} (1991),  653--678. 
\doi{10.2307/2944359}

\bibitem{Kho}
 A. G. Khovanskii,
 \textit{Newton polyhedra and toroidal varieties.}
{Funct. Anal. Appl.} \textbf{11} (1978), 289--296.
\doi{10.1007/BF01077143}
 
\bibitem{Kou}  
A.G. Kouchnirenko,  
\textit{Poly\`edres de Newton et nombres de Milnor. } 
Invent. Math. {\bf32} (1976), 1--31. 
\doi{10.1007/BF01389769}
 
\bibitem{Kou2}
\selectlanguage{russian}
А. Г. Кушниренко, Критерий существования невырожденной 
квазиоднородной функции с заданными весами,
УМН, 1977, том 32, выпуск 3(195), 169--170.\\
\selectlanguage{english}%
\href{http://mi.mathnet.ru/umn3200}{http://mi.mathnet.ru/umn3200}

\bibitem{lu}  
Ignacio Luengo,  
\textit{The $\mu$-constant stratum is not smooth.} 
Invent. Math. {\bf90} (1987), 139--152. 
\doi{10.1007/BF01389034}

\bibitem{luc}  
Ignacio Luengo, 
\textit{On the Existence of Complete Families of Projective Plane Curves, which are Obstructed.}  J. Lond. Math. Soc., II. Ser. \textbf{36} (1987),  33--43. 
\doi{10.1112/jlms/s2-36.1.33}

\bibitem{MW}
A. Melle-Hern\'andez and C. T. C. Wall,
\textit{Pencils of curves on smooth surfaces.}
Proc. Lond. Math. Soc. (3) \textbf{83} (2001), 257--278.
\doi{10.1112/plms/83.2.257}

\bibitem{Mo-bk}
Pinaki Mondal,
\textit{How many zeroes? Counting the number of solutions of systems of polynomials via geometry at infinity.}
Preprint, arXiv:1806.05346v2

\bibitem{Ger}
Angel Montesinos,
{\sc Gérmenes}.
Available at 
\href{http://www.uv.es/montesin}{http://www.uv.es/montesin}  (2009). 


\bibitem{Ol}
G. Oleksik
\textit{The \L ojasiewicz exponent of nondegenerate surface singularities.}
Acta Math. Hungar., {\bf138} ({\bf1--2}) (2013), 179--199.
\doi{10.1007/s10474-012-0285-5}

\bibitem{pe} 
Ruud Pellikaan, 
\textit{Finite determinacy of functions with non-isolated singularities.} 
Proc. London Math. Soc. s3-{\bf57} (1988), 357--382.  
\doi{10.1112/plms/s3-57.2.357}

\bibitem{dp}
Andrew A. du Plessis,
\textit{Minimal Intransigent Hypersurfaces.}
In: Real and Complex Singularities
(Trends in Mathematics), 
Basel: Birkh\"auser, pp. 299--310 (2006).
\doi{10.1007/978-3-7643-7776-2_21}
 
\bibitem{se} 
Peter Seibt, 
\textit{Differential filtrations and symbolic powers of regular primes.} 
Math. Z. 166 (1979), 159--164. 
 \doi{10.1007/BF01214042}
  

\bibitem{st} Jan Stevens,  
\textit{On the $\mu$-constant stratum and the 
    $V$-filtration: an example.} 
    Math. Z. {\bf 201} (1989), 139--144. 
 \doi{10.1007/BF01162001} 
    
\bibitem{Teiss}
Bernard Teissier,
\textit{Cycles évanescents, sections planes et conditions de Whitney.} In: Singularités à Cargèse, pp. 285--362. Asterisque,  \textbf{7} et \textbf{8}, Soc. Math. France, Paris, 1973.
 
 
\bibitem{Te} 
Bernard Teissier,  
Appendix  to : O. Zariski, Le probl\`eme des modules pour les branches planes.  Course given at the Centre de Math\'ematiques de l'Ecole Polytechnique, Paris, October--November 1973. Second edition. Hermann, Paris, 1986.   

 
\bibitem{Va} 
A.N. Varchenko,  
\textit{Zeta-function of monodromy and Newton's diagram. } 
Invent. Math. {\bf37} (1976), 253--262. 
\doi{10.1007/BF01390323}
 
\bibitem{Wa} 
 C.T.C. Wall, 
\textit{Newton polytopes and non-degeneracy. } 
J. Reine Angew. Math. {\bf509} (1999), 1--19. 
\doi{10.1515/crll.1999.509.1}

\bibitem{Zar}
Oscar Zariski,
\textit{Algebraic surfaces.}
Second supplemented edition.
With appendices by S. S. Abhyankar, J. Lipman, and D. Mumford.
Springer-Verlag, Berlin-Heidelberg-New York,  1971.
\doi{10.1007/978-3-642-61991-5}

\end{thebibliography}
\end{document}